\documentclass{amsart}
\usepackage[all]{xy}
\usepackage{color}
\usepackage{amsthm}
\usepackage{amssymb}
\usepackage[colorlinks=true]{hyperref}

\numberwithin{equation}{section}

\newtheorem{theorem}[equation]{Theorem}
\newtheorem*{theorem*}{Theorem}
\newtheorem{lemma}[equation]{Lemma}
\newtheorem{proposition}[equation]{Proposition}
\newtheorem{corollary}[equation]{Corollary}
\newtheorem*{corollary*}{Corollary}

\theoremstyle{remark}
\newtheorem{definition}[equation]{Definition}

\newtheorem{example}[equation]{Example}

\theoremstyle{remark}
\newtheorem{remark}[equation]{Remark}

\newcommand{\cA}{{\mathcal A}}
\newcommand{\cB}{{\mathcal B}}
\newcommand{\cC}{{\mathcal C}}
\newcommand{\cD}{{\mathcal D}}
\newcommand{\cE}{{\mathcal E}}
\newcommand{\cF}{{\mathcal F}}
\newcommand{\cG}{{\mathcal G}}

\newcommand{\cK}{{\mathcal K}}

\newcommand{\cO}{{\mathcal O}}
\newcommand{\cP}{{\mathcal P}}

\newcommand{\cU}{{\mathcal U}}

\newcommand{\bbC}{\mathbf{C}}

\newcommand{\bbL}{\mathbf{L}}

\newcommand{\bbP}{\mathbf{P}}
\newcommand{\bbR}{\mathbf{R}}

\newcommand{\bbQ}{\mathbf{Q}}
\newcommand{\bbZ}{\mathbf{Z}}

\DeclareMathOperator{\NChow}{NChow} 
\DeclareMathOperator{\Chow}{Chow} 

\DeclareMathOperator{\HH}{HH}

\DeclareMathOperator{\HP}{HP}
\DeclareMathOperator{\CH}{CH}

\DeclareMathOperator{\Fun}{Fun} 

\newcommand{\dgcat}{\mathsf{dgcat}}

\newcommand{\perf}{\mathsf{perf}}

\newcommand{\dg}{\mathsf{dg}}

\newcommand{\Hom}{\mathrm{Hom}}
\newcommand{\End}{\mathrm{End}}

\newcommand{\rep}{\mathsf{rep}}

\newcommand{\per}{\mathrm{per}}
\newcommand{\unit}{\mathbf{1}}
\newcommand{\todd}{\mathrm{Td}}
\newcommand{\sqtodd}{\sqrt{\todd}}
\newcommand{\chern}{\mathit{ch}}

\newcommand{\Hmo}{\mathsf{Hmo}}
\newcommand{\op}{\mathsf{op}}

\newcommand{\too}{\longrightarrow}

\newcommand{\ie}{\textsl{i.e.}\ }

\begin{document}

\title[Lefschetz and Hirzebruch-Riemann-Roch formulas]{Lefschetz and Hirzebruch-Riemann-Roch formulas \\via noncommutative motives}

\author{Denis-Charles Cisinski}
\address{Universit\'e Paul Sabatier\\
Institut de Math\'ematiques de Toulouse\\
118 route de Narbonne\\
F-31062 Toulouse Cedex 9}
\email{denis-charles.cisinski@math.univ-toulouse.fr}
\urladdr{http://www.math.univ-toulouse.fr/~dcisinsk/}

\author{Gon{\c c}alo~Tabuada}

\address{Gon{\c c}alo Tabuada, Department of Mathematics, MIT, Cambridge, MA 02139, USA}
\email{tabuada@math.mit.edu}
\urladdr{http://math.mit.edu/~tabuada/}

\subjclass[2000]{14A22, 14C20, 14F25, 14F40, 18D20, 19L10}
\date{\today}

\thanks{Gon{\c c}alo Tabuada was partially supported by the NEC Award-2742738.}


\abstract{V.~Lunts has recently established Lefschetz fixed point theorems for Fourier-Mukai
functors and dg algebras. In the same vein, D.~Shklyarov introduced the noncommutative analogue
of the Hirzebruch-Riemman-Roch theorem. In this short article, we see how these constructions
and computations formally stem from their motivic counterparts.}
}
\maketitle
\vskip-\baselineskip
\vskip-\baselineskip
\section{Introduction: traces, pairings, and Hochschild homology}
Let $k$ be a field. Lunts' results on the noncommutative Lefschetz formula
can be summarised as follows:
\begin{theorem}{(see~\cite[Thms 1.1, 1.2 and 1.4]{Lunts})}\label{thm:Lunts1}
Let $X$ be a smooth and proper $k$-scheme and $\cE$ a bounded
complex of coherent $\cO_{X\times X}$-modules. Then, the following equality holds
\begin{equation}\label{eq:equality1}
\sum_i (-1)^i \, \mathrm{dim}\,  \HH_i(\cE) = \sum_j (-1)^j \,
\mathrm{Tr}\, \HH_j(\Phi_\cE)\,,
\end{equation}
where $\Phi_\cE$ is the associated Fourier-Mukai functor and
$\HH(\cE)=\HH(X,\cE)$ the Hochschild homology of $X$ with coefficients in $\cE$.
Moreover, when $k=\bbC$ we have
\begin{equation}\label{eq:equality2}
\sum_i (-1)^i \, \mathrm{dim} \, \HH_i(\cE) =
\mathrm{Tr}\,  H^{\mathrm{even}}(\Phi_\cE) - \mathrm{Tr} \, H^{\mathrm{odd}}(\Phi_\cE)\,, 
\end{equation}
where $H^{\mathrm{even}}$ and $H^{\mathrm{odd}}$ are the even and odd parts of Betti cohomology
with rational coefficients and $H^\ast(\Phi_\cE):H^\ast(X)\to H^\ast(X)$ the correspondence associated to the class $\chern(\cE)\cdot\sqtodd_X \in H^\ast(X\times X)$. 

Let $A$ be a smooth and proper dg $k$-algebra and $M$ a perfect dg $A$-bimodule. Then, the following equality holds
\begin{equation}\label{eq:equality3}
\sum_i (-1)^i \, \mathrm{dim} \, \HH_i(A;M) =
\sum_j (-1)^j\, \mathrm{Tr}\, \HH_j(\Phi_M)\,.
\end{equation}
\end{theorem}
Shklyarov's noncommutative analogue of the Hirzebruch-Riemman-Roch theorem
can be stated as follows.
\begin{theorem}{(see~\cite[Thm.~3]{Shklyarov})}\label{thm:HRR}
 Let $A$ be a proper dg $k$-algebra. Then, there is a canonical pairing
$$ \langle -,- \rangle: \HH(A)\otimes_k\HH(A^\op)\too k$$
such that for any two perfect dg $A$-modules $M$ and $N$
the following equality holds
\begin{equation}\label{eq:HRR}
\chi(M,N) = \langle \chi(N), \chi(DM)\rangle\,,
\end{equation}
where $DM$ is the formal dual of $M$, $\chi$ is the Euler character
(also known as Dennis trace map), and
$\chi(M,N)=\chi(\bbR\mathit{Hom}_A(M,N))$ is the Euler character of the derived Hom
of maps from $M$ to $N$.
\end{theorem}
In this short article we prove the following general motivic results (over a base commutative ring $k$) and show how the above equalities~\eqref{eq:equality1}-\eqref{eq:equality3} and \eqref{eq:HRR} can be recovered from them.
\begin{theorem}[Noncommutative Lefschetz]\label{thm:main}
Let $\cA$ be a smooth and proper dg category (see \S\ref{sec:dg}),
$M$ a perfect dg $\cA$-bimodule, and $L:\dgcat(k) \to \mathsf{D}$ a
symmetric monoidal\footnote{As far as monoidal categories and monoidal functors
are concerned, we shall use the terminology of MacLane's book~\cite{ML}.}
additive invariant (see Definition~\ref{def:tensor-additive2})
which becomes strongly monoidal when restricted to smooth and proper dg categories.
Assume that the base ring $k$ is local (or more generally that $K_0(k)=\bbZ$).
Then, the following equality holds in the ring of endomorphims of the unit object
of $\mathsf{D}$:
\begin{equation}\label{eq:equality-main}
\sum_i(-1)^i \, \mathrm{rk}\,  \HH_i(\cA;M) = \mathrm{Tr}\, L(\Phi_M)\,.
\end{equation}
\end{theorem}
\begin{theorem}\label{thm:main2}
Let $k$ be a field, $X$ a smooth and proper $k$-scheme, and $\cE$ a bounded complex of coherent $\cO_{X\times X}$-modules.
Then, for every Weil cohomology $H^\ast$ defined on smooth and proper
$k$-schemes, we have the equality
\begin{equation}\label{eq:equality-main2}
\sum_i (-1)^i \, \mathrm{dim} \, \HH_i(\cE)
= \mathrm{Tr}\,  H^{\mathrm{even}}(\Phi_\cE) - \mathrm{Tr}\,  H^{\mathrm{odd}}(\Phi_\cE)\, ,
\end{equation}
where $H^{\mathrm{even}}$ and $H^{\mathrm{odd}}$ are the
the even and odd parts of $H^\ast$.
\end{theorem}

\begin{theorem}[Noncommutative Hirzebruch-Riemman-Roch]\label{thm:HRR1}
Let $L:\dgcat(k)\to \mathsf{D}$ be
a strongly symmetric monoidal additive invariant (see Definition~\ref{def:tensor-additive2}),
and $\unit$ the unit object of $\mathsf{D}$.
For any dg category $\cA$, there is a canonical Chern character map
$$\chern_L:K_0(\cA)\too \Hom_\mathsf{D}(\unit,L(\cA))\, .$$
Moreover, in the case where $\cA$ is proper, there is a canonical pairing in $\mathsf{D}$
$$L(\cA)\otimes L(\cA^\op)\too\unit$$
and consequently a canonical pairing of abelian groups
$$\langle -,- \rangle:\Hom_\mathsf{D}(\unit,L(\cA))\otimes
\Hom_\mathsf{D}(\unit,L(\cA^\op))\to\End_\mathsf{D}({\bf 1})\,.$$
Furthermore, for any two perfect dg $\cA$-modules $M$ and $N$, the following equality
\begin{equation}\label{eq:HRR1}
ch_L(\bbR\mathit{Hom}_\cA(M,N))=\langle ch_L(N), ch_L(DM)\rangle
\end{equation}
holds in $\End_{\mathsf{D}}({\bf 1})$.
\end{theorem}
Consult Sections \ref{section:nclefschetz} and \ref{sec:HRR}
for the proofs of these three statements. The proofs of Theorems \ref{thm:main} and \ref{thm:HRR1} rely on formal arguments, namely on the fact that symmetric monoidal additive invariants correspond to
symmetric monoidal functors from an adequate category of noncommutative motives, whose
Hom's can be understood as $K_0$-groups; see Proposition \ref{prop:universalmonoidal}.
The proof of Theorem \ref{thm:main2}
is also quite formal, once we have related the classical category of Chow motives
with the aforementioned category of noncommutative motives. This is
just a categorical reformulation of the Grothendieck-Riemann-Roch theorem; see
Theorem \ref{thm:categoricalGRR}. The proof of the equalities
\eqref{eq:equality-main}, \eqref{eq:equality-main2} and \eqref{eq:HRR1} follows also the same pattern: the left-hand-sides are numbers obtained from a motivic construction (trace or pairing
in a suitable category of motives) which one manages to compute using
the functorialities of the theory of dg categories. The right-hand-sides are similarly obtained from the realisation of the original data. In conclusion, these equalities follow simply from the fact that realizations of motives are strongly symmetric monoidal, \ie they satisfy K\"unneth formulas.

\subsection*{Acknowledgments} The authors are very grateful to the anonymous referee for all his/her
comments that made us improve this short article.

\subsection*{Notations}\label{sec:notations}
The letter $k$ will always stand for a commutative ring with unit $k$.
Given a quasi-compact separated $k$-scheme $X$, we will denote by $\perf(X)$ the dg category
of perfect complexes of $\cO_X$-modules, and by $\cD_\perf(X)$ the corresponding
triangulated category. Recall that if $X$ is regular then every bounded complex of coherent sheaves on
$X$ is perfect (up to quasi-isomorphism).

Given an essentially small category $\cC$, we will write $\mathrm{Iso}\, \cC$ for the
set of isomorphism classes of objects of $\cC$.

\section{Differential graded categories}\label{sec:dg}
Let $k$ be a base commutative ring and $\cC(k)$ the category of dg $k$-modules. A {\em differential graded (=dg) category} $\cA$ is a category enriched over $\cC(k)$ (morphism sets $\cA(x,y)$ are dg $k$-modules) in such a way that composition fulfills the Leibniz rule: $d(f\circ g) = d(f) \circ g + (-1)^{\mathrm{deg}(f)}f \circ d(g)$. A {\em dg functor} $F:\cA \to \cB$ is a functor enriched over $\cC(k)$; consult Keller's ICM address~\cite[\S2.2-2.3]{ICM}. In what follows we will write $\dgcat(k)$ for the category of (small) dg categories and dg functors. 


Let $\cA$ be a dg category. Its {\em opposite} dg category $\cA^\op$ has the same objects and dg
$k$-modules of morphisms given by $\cA^\op(x,y):=\cA(y,x)$. A (right) dg {\em $\cA$-module} is a
dg functor $M:\cA^\op \to \cC^\dg(k)$ with values in the dg category $\cC^\dg(k)$ of dg $k$-modules.
Let us denote by $\widehat{\cA}$ the category of dg $\cA$-modules; see \cite[\S2.3]{ICM}.
Recall from \cite[\S3.2]{ICM} that the {\em derived category $\cD(\cA)$ of $\cA$} is the
localization of $\widehat{\cA}$ with respect to the class of objectwise quasi-isomorphisms.
Its full subcategory of compact objects (see \cite[Def.~4.2.7]{Neeman}) will be denoted by
$\cD_c(\cA)$. A dg functor $F:\cA \to \cB$ is called a {\em Morita equivalence} if the restriction
of scalars functor $\cD(\cB) \stackrel{\sim}{\to} \cD(\cA)$ is an equivalence of triangulated
categories; see \cite[\S4.6]{ICM}.

Let $\Hmo(k)$ be the localization of $\dgcat(k)$ with respect to the class of Morita equivalences.
The tensor product of dg categories can be naturally derived $-\otimes^\bbL-$ giving thus rise to a
symmetric monoidal structure on $\Hmo(k)$; see \cite[Remark~5.11]{IMRN}.
A {\em $\cA\text{-}\cB$-bimodule $M$} is a dg functor $M:\cA \otimes^\bbL \cB^\op\to \cC^\dg(k)$,
\ie a dg $(\cA^\op \otimes^\bbL \cB)$-module.
Given any two dg categories
$\cA$ and $\cB$, we have a bijection 
\begin{eqnarray}\label{eq:bij}
\mathrm{Iso}\,\rep(\cA,\cB) \stackrel{\sim}{\too} \Hom_{\Hmo(k)}(\cA,\cB) &&
M \longmapsto \left(x\longmapsto \cA(-,x)\otimes^\bbL_\cA M \right)\,,
\end{eqnarray}
where $\rep(\cA,\cB)$ is the full triangulated subcategory of $\cD(\cA^\op \otimes^\bbL\cB)$ spanned by the dg
$\cA\text{-}\cB$-bimodules $M$ such that for every object $x \in \cA$ the associated dg $\cB$-module $M(-,x)$ belongs
to $\cD_c(\cB)$; consult \cite[Cor.~5.10]{IMRN} for further details. Moreover, under the above bijection \eqref{eq:bij}, the composition law in $\Hmo(k)$ corresponds to the (derived) tensor product of bimodules. 

Following Kontsevich~\cite{IAS,ENS}, a dg category $\cA$ is called  {\em smooth} if it is perfect as a
dg bimodule over itself. A dg category $\cA$ is said to be {\em proper} if for every ordered pair of objects $(x,y)$ the dg $k$-module $\cA(x,y)$ is perfect\footnote{If the
base ring is not a field, this notion of properness really belongs to the realm of
derived geometry; perfect complexes over flat and proper $k$-schemes still give
examples, though.}.
If $\cA$ is smooth and proper we have an equivalence of triangulated categories
$\rep(\cA,\cB)\simeq \cD_c(\cA^\op \otimes^\bbL \cB)$.
Consequently, \eqref{eq:bij} reduces to the bijection
\begin{eqnarray*}
\mathrm{Iso}\, \cD_c(\cA^\op \otimes^\bbL \cB) \stackrel{\sim}{\too} \Hom_{\Hmo(k)}(\cA,\cB)&&
M\longmapsto \Phi_M:=\left(x\longmapsto \cA(-,x)\otimes^\bbL_\cA M \right) \,.
\end{eqnarray*}
In fact, smooth and proper dg categories are precisely the
rigid objects in $\Hmo(k)$; see \cite[Theorem~4.8]{CT1}. More precisely, if $\cA$ is smooth and proper then $\cA^\op$ is its dual and for any two dg categories $\cB$ and $\cC$ one has
$$\Hom_{\Hmo(k)}(\cA\otimes^\bbL\cB,\cC)\simeq
\Hom_{\Hmo(k)}(\cB,\cA^\op\otimes^\bbL\cC)\, .$$
For instance, for any smooth and proper $k$-scheme $X$, the dg category $\perf(X)$
is smooth and proper (this follows immediately from To\"en's
beautiful results \cite[Lemma 8.11 and Theorem 8.15]{Toen}).
\section{Category of integral kernels}\label{sec:category-kernels}
The category $\mathbf{P}(k)$ of {\em integral kernels} is the full subcategory of $\Hmo(k)$
whose objects are the dg categories isomorphic to $\perf(X)$ for any smooth and proper
$k$-scheme $X$. Let us choose, for each smooth and proper $k$-scheme $X$,
and equivalence
\begin{eqnarray}\label{eq:isom0}
\perf(X)^\op\simeq\perf(X)
\end{eqnarray}
(that is, let us identify $\perf(X)$ with its dual). We have then, for any two
smooth and proper $k$-schemes $X$ and $Y$, an identification
of the form
\begin{eqnarray}\label{eq:isom00}
\perf(X)^\op\otimes^\bbL\perf(Y)\simeq\perf(X)\otimes^\bbL\perf(Y)
\simeq\perf(X\times Y)
\end{eqnarray}
where the latter isomorphism follows automatically from \cite[Lemma 8.11]{Toen}
for instance.
Using \eqref{eq:bij}, we thus get a natural bijection
\begin{eqnarray}\label{eq:isom000}
\Hom_{\Hmo(k)}(\perf(X),\perf(Y))\simeq\mathrm{Iso}\,\cD_\perf(X\times Y)\, . 
\end{eqnarray}
In the case where \eqref{eq:isom0} is induced by the
duality functor $\cE\mapsto \bbR\mathit{Hom}(\cE,\cO_X)$,
the bijection \eqref{eq:isom000} can be described explicitely as
\begin{eqnarray}\label{eq:isom1}
\mathrm{Iso}\,\cD_\perf(X\times Y) \stackrel{\sim}{\too}
\Hom_{\Hmo(k)}(\perf(X),\perf(Y)) && \cF \mapsto \Phi_\cF\,,
\end{eqnarray}
where $\Phi_\cF$ is the {\em Fourier-Mukai dg functor}
\begin{eqnarray}\label{eq:isom2}
\Phi_\cF: \perf(X) \too \perf(Y) && \cG \mapsto \bbR(\pi_Y)_\ast(\pi^\ast_X(\cG)\otimes^\bbL \cF)
\end{eqnarray}
(see To\"en's proof of \cite[Theorem 8.15]{Toen}).
If one prefers to define \eqref{eq:isom0} using
Grothendieck-Serre duality, that is via the functor $\cE\mapsto \bbR\mathit{Hom}(\cE,\cK_X)$ with $\cK_X$ the dualizing complex, then the bijection \eqref{eq:isom000}
gives the description of the category $\mathbf{P}(k)$ considered by
Caldararu and Willerton in \cite{Cald} for instance\footnote{The
category $\mathbf{P}(k)$ is in fact the truncation of the bicategory
studied in \cite{Cald} (in the case where $k$ is a field).}.
Note however that the choice of the isomorphisms \eqref{eq:isom0} does not really matters,
as any two duals of $\perf(X)$ are uniquely isomorphic (as duals),
so that the corresponding
theories of pairings and trace maps cannot be seriously affected by any such
choices\footnote{For the reader who prefers to have explicit computations, at least
through the lens of Hochschild homology, we refer to the work of Ramadoss \cite{Ramadoss}.}.
In what follows, we will thus
work using the explicit identification \eqref{eq:isom1}, which corresponds to the convention
adopted in the work of To\"en~\cite{Toen} and Lunts~\cite{Lunts}. In this case, we can understand
the structure of monoidal category of $\mathbf{P}(k)$ as follows.
Given perfect complexes $\cF\in \cD_\perf(X\times Y)$ and $\cG \in \cD_\perf(Y\times Z)$,
their composition is the perfect complex
$$\bbR(\pi_{XZ})_\ast (\pi^\ast_{XY}(\cF) \otimes^\bbL \pi^\ast_{YZ}(\cG))\in \cD_\perf(X\times Z)\, ,$$
(where $\pi_{ST}$ denotes the canonical projection from $X\times Y\times Z$ to $S\times T$).
The identity of an object $\perf(X)$ of $\mathbf{P}(k)$ is given by the structure
sheaf $\bbR\Delta_\ast(\cO_X) \in \cD_\perf(X\times X)$ of the diagonal $\Delta$ of $X\times X$.
Given perfect complexes $\cE \in \cD_\perf(X\times W)$ and $\cF \in \cD_\perf(Y\times Z)$,
their product is
$$ \cE \boxtimes\cF \in \cD_\perf(X\times Y \times W \times Z)\,.$$
This explicit description of $\mathbf{P}(k)$ will be mainly used
to make the link between noncommutative Chow motives and classical intersection
theory; see Theorem \ref{thm:categoricalGRR}.
\section{Noncommutative (Chow) motives}\label{sec:Chow}
Let $\Hmo_0(k)$ be the additive category with the same objects as $\Hmo(k)$ and with morphisms given by $\Hom_{\Hmo_0(k)}(\cA,\cB):= K_0 \rep(\cA,\cB)$ where $K_0\rep(\cA,\cB)$ is the Grothendieck group of the triangulated category $\rep(\cA,\cB)$. The composition law is induced from the one on $\Hmo(k)$; consult \cite[\S6]{IMRN}. The derived tensor product on $\Hmo(k)$ extends by linearity to $\Hmo_0(k)$ giving rise to a symmetric monoidal structure. Moreover, there is a natural sequence of symmetric monoidal functors
\begin{equation}\label{eq:universal}
\cU: \dgcat(k) \too \Hmo(k) \stackrel{[-]}{\too} \Hmo_0(k)\,.
\end{equation}
The first one is the identity on objects and sends a dg functor $F:\cA \to \cB$ to the corresponding $\cA\text{-}\cB$-bimodule. The functor $[-]$ is also the identity on objects and sends a $\cA\text{-}\cB$-bimodule $M$ to its class $[M]$ in the Grothendieck group $K_0\rep(\cA,\cB)$.

The category $\NChow(k)$ of {\em noncommutative (Chow) motives} is by definition the idempotent completion of full subcategory of $\Hmo_0(k)$ of the smooth and proper dg categories. Note that we have the following description of its morphism
sets
$$\Hom_{\NChow(k)}(\cA,\cB)=K_0(\cD_c(\cA^\op \otimes^\bbL\cB))=K_0(\cA^\op \otimes^\bbL \cB)\,.$$ 
Note also that since the smooth and proper dg categories are stable under (derived) tensor product, the category $\NChow(k)$ is symmetric monoidal. Since the rigid objects in the symmetric monoidal category $\Hmo(k)$ are precisely the smooth and proper dg categories, $\NChow(k)$ is a rigid tensor category.
\section{Symmetric monoidal additive invariants}\label{sec:symmetric}
Let $\cA$ be a dg category. Consider the dg category $T(\cA)$ whose objects are the pairs $(i,x)$,
with $i \in \{1,2\}$ and $x$ an object of $\cA$. The dg $k$-module
$T(\cA)((i,x),(i',x'))$ is equal to $\cA(x,x')$ if $i'\geq i$ and is $0$ otherwise.
Composition is induced from $\cA$; consult \cite[\S4]{IMRN} for details.
Note that we have two natural inclusion dg functors
\begin{eqnarray*}
\iota_1 : \cA \too T(\cA) && \iota_2: \cA \too T(\cA)\,.
\end{eqnarray*}

\begin{definition}\label{def:tensor-additive}
Let $L:\dgcat(k) \to \mathsf{D}$ be a functor with values in an additive category. We say that $L$ is an {\em additive invariant} if it satisfies the following two conditions:
\begin{itemize}
\item[(i)] it sends Morita equivalences to isomorphisms;
\item[(ii)] the above inclusion dg functors $\iota_1$ and $\iota_2$ induce an isomorphism\footnote{Condition (ii) can be equivalently formulated in terms of semi-orthogonal decompositions in the sense of Bondal-Orlov; see \cite[Thm.~6.3(4)]{IMRN}.}
$$ [L(\iota_1)\,\, L(\iota_2)]: L(\cA) \oplus L(\cA) \stackrel{\sim}{\too} L(T(\cA))\,.$$
\end{itemize}
\end{definition}
\begin{theorem}{(\cite[Thms. 5.3 and 6.3]{IMRN})}\label{thm:universal}
The above functor \eqref{eq:universal} is the {\em universal additive invariant}, \ie given any additive category $\mathsf{D}$ we have an induced equivalence of categories
\begin{equation}\label{eq:induced}
\cU^\ast: \Fun_{\mathsf{add}}(\Hmo_0(k),\mathsf{D}) \stackrel{\sim}{\too} \Fun_{\mathsf{A}}(\dgcat(k),\mathsf{D})\,,
\end{equation}
where the left-hand-side denotes the category of additive functors from $\Hmo_0(k)$ to $\mathsf{D}$ and the right-hand-side the category of additive invariants with values in $\mathsf{D}$.
\end{theorem}

Note that $\Hmo_0(k)$ is a symmetric monoidal category such that the
canonical functor $\Hmo(k)\to\Hmo_0(k)$ is strictly symmetric monoidal.
The preceding theorem has a monoidal version as follows.
We will say that a dg category $\cA$ is \emph{flat} if, for any
objects $x$ and $y$ of $\cA$, the dg $k$-module $\cA(x,y)$
is a flat $k$-module in each degree. We will write $\dgcat_f(k)$
for the full subcategory of $\dgcat(k)$ whose objects are the flat
dg categories.

\begin{definition}\label{def:tensor-additive2}
A \emph{symmetric monoidal additive invariant} is an additive invariant
$L:\dgcat(k)\to\mathsf{D}$ with values in a symmetric monoidal category $\mathsf{D}$
together with a structure of symmetric monoidal functor on the restriction of $L$
to the subcategory $\dgcat_f(k)$. Such an additive invaiant $L$ is said
to be \emph{strongly symmetric monoidal when restricted to
smooth and proper dg categories} if its restriction to the
full subcategory of $\dgcat_f(k)$ which consists of flat smooth and
proper dg categories is strongly monoidal in the sense of MacLane \cite{ML}.
\end{definition}

\begin{proposition}\label{prop:universalmonoidal}
Let $L:\dgcat(k)\to\mathsf{D}$ be a symmetric monoidal additive invariant.
Then, the corresponding additive functor $\underline{L}:\Hmo_0(k)\to\mathsf{D}$
has a canonical structure of symmetric monoidal functor.
In particular, if $L$ is moreover strongly symmetric monoidal when restricted to
smooth and proper dg categories, then the restriction of
$\underline{L}$ to the category $\NChow(k)$ is strongly monoidal.
\end{proposition}

\begin{proof}
Let $n\geq 1$ be an integer and $\mathsf{D}$ an additive category.
We claim that the composing with the functor $\cU^n:\dgcat(k)^n\to\Hmo_0(k)$
induces an equivalence of categories
\begin{equation}\label{eq:multiinduced}
\Fun^n_{\mathsf{add}}(\Hmo_0(k),\mathsf{D}) \stackrel{\sim}{\too}
\Fun^n_{\mathsf{A}}(\dgcat(k),\mathsf{D})\, ,
\end{equation}
where the left-hand-side denotes the category of functors $\Hmo_0(k)^n\to\mathsf{D}$ which
are additive in each variable, and the right-hand-side the category of functors $\dgcat(k)^n\to\mathsf{D}$ which are additive invariants
in each variable.
We claim also that the inclusion $\dgcat_f(k)\subset\dgcat$ induces equivalences of categories
\begin{equation}\label{eq:multiinduced2}
\Fun^n_{\mathsf{A}}(\dgcat(k),\mathsf{D}) \stackrel{\sim}{\too}
\Fun^n_{\mathsf{A}}(\dgcat_f(k),\mathsf{D})\, ,
\end{equation}
where $\Fun^n_{\mathsf{A}}(\dgcat_f(k),\mathsf{D})$ denotes the category of
functors $\dgcat_f(k)^n\to\mathsf{D}$ which satisfy conditions (i)-(ii)
of Definition \ref{def:tensor-additive} (this makes sense because $T(\cA)$ is flat
whenever $\cA$ is flat). Note that the proof of Proposition~\ref{prop:universalmonoidal} follows automatically from the equivalences \eqref{eq:multiinduced}-\eqref{eq:multiinduced2}. The equivalences \eqref{eq:multiinduced2}
follow immediately from the compatibility of the functor $T$ with Morita equivalences
and from the fact that the inclusion functor
$\dgcat_f(k)\subset\dgcat(k)$ induces an equivalence after localization by Morita
equivalences since every cofibrant dg category in the sense of the
model category structures of \cite{CRAS,IMRN} is flat. In order to finish the proof it thus remains only to check that \eqref{eq:multiinduced} are equivalences of categories. We proceed by induction on $n$. The case where $n=1$ is provided by Theorem \ref{thm:universal}.
If $n>1$, then we have the following obvious equivalences of categories.
$$\begin{aligned}
\Fun^n_{\mathsf{add}}(\Hmo_0(k),\mathsf{D})
&\simeq\Fun_{\mathsf{add}}(\Hmo_0(k),\Fun^{n-1}_{\mathsf{add}}(\Hmo_0(k),\mathsf{D}))\\
\Fun^n_{\mathsf{A}}(\dgcat(k),\mathsf{D})
&\simeq\Fun_{\mathsf{A}}(\dgcat(k),\Fun^{n-1}_{\mathsf{A}}(\dgcat(k),\mathsf{D}))\, .
\end{aligned}$$
On the other hand, by induction (and by applying again Theorem \ref{thm:universal})
we have
$$\begin{aligned}
\Fun_{\mathsf{add}}(\Hmo_0(k),\Fun^{n-1}_{\mathsf{add}}(\Hmo_0(k),\mathsf{D}))
&\simeq\Fun_{\mathsf{A}}(\dgcat(k),\Fun^{n-1}_{\mathsf{add}}(\Hmo_0(k),\mathsf{D}))\\
&\simeq\Fun_{\mathsf{A}}(\dgcat(k),\Fun^{n-1}_{\mathsf{A}}(\dgcat(k),\mathsf{D}))\,.
\end{aligned}$$
This achieves the proof.
\end{proof}
\begin{remark}
Note that the proof of the Proposition~\ref{prop:universalmonoidal} is also the proof
of the existence of the canonical symmetric monoidal structure
on $\Hmo_0(k)$.
\end{remark}
Let us now describe some examples of additive invariants which are moreover
symmetric monoidal.
\begin{example}[Hoschchild homology]\label{example:HH}
Let $\cD(k)$ be the derived category of the base commutative ring $k$. By construction this category is additive and symmetric monoidal. As explained in \cite[\S6.1]{IMRN}, the Hochschild homology functor
\begin{equation*}
\HH: \dgcat(k) \too \cD(k)
\end{equation*}
is an additive invariant. Thanks to \cite[Example~7.9]{CT1},
$\HH$ is furthermore symmetric monoidal and hence defines a
symmetric monoidal functor
\begin{equation*}
\HH: \Hmo_0(k) \too \cD(k)\, .
\end{equation*}
In particular, for any smooth and proper dg category $\cA$ (= rigid object of $\Hmo_0$ with dual $\cA^\op$), $\HH(\cA)$ is a rigid object of $\cD(k)$ (\ie a perfect complex) and the pairing considered in Theorem \ref{thm:HRR}
is non degenerate.

For a $k$-scheme $X$, we will write
$\HH(X)=\HH(\perf(X))$. The restriction of the functor $\HH$
to the category $\mathbf{P}(k)$
gives the usual Hochschild homology of smooth and proper $k$-schemes, whatever
version the reader might prefer (the agreement with Weibel's version of Hochschild
homology of schemes (with coefficients) is proved by Keller in \cite{Ringed},
while, as explained in~\cite[\S4.2]{Cald}, Weibel's definition is isomorphic to
Caldararu-Willerton's version of Hochschild homology).
\end{example}
\begin{example}[Mixed complexes]\label{example:mix}
Following Kassel \cite[\S1]{Kassel}, a {\em mixed complex} $(N,b,B)$ is a $\bbZ$-graded $k$-module $\{N_n\}_{n \in \bbZ}$ endowed with a degree $+1$ endomorphism $b$ and a degree $-1$ endomorphism $B$ satisfying the relations $b^2=B^2=Bb +bB=0$. Equivalently, a mixed complex is a right dg module over the dg algebra $\Lambda:=k[\epsilon]/\epsilon^2$, where $\epsilon$ is of degree $-1$ and $d(\epsilon)=0$. Let $\cD(\Lambda)$ be the derived category of mixed complexes. By construction, $\cD(\Lambda)$ is additive and as explained in \cite[\S6.1]{IMRN} the mixed complex functor
\begin{equation}\label{eq:C}
C: \dgcat(k) \too \cD(\Lambda)
\end{equation}
is an additive invariant. Moreover, $\cD(\Lambda)$ carries a symmetric monoidal structure defined on the underlying dg $k$-modules under which $C$ is symmetric monoidal; see \cite[Example~7.10]{CT1}.
\end{example}
\begin{example}[Periodic cyclic homology]\label{Example:HP}
Assume that $k$ is a field. Using the work of Kassel~\cite{Kassel},
one can understand periodic cyclic homology as follows (we refer
to \cite[Examples 8.11 and 9.11]{CT1} for a slightly more detailed exposition).
Let $k[u]$ be the cocommutative Hopf algebra
of polynomials in one variable $u$ in degree $2$. To a mixed complex
(that is a $\Lambda$-module) $M$, one associates its periodization $P(M)$. This is
the $k[u]$-comodule whose underlying complex is $k\otimes^\bbL_\Lambda M$ and whose
comultiplication by $u$ is induced by Connes' map
$$S:k\otimes^\bbL_\Lambda M[-2]\too k\otimes^\bbL_\Lambda M\, .$$
If $\cD(k[u]\text{-}\mathrm{Comod})$ denotes the homotopy category of
$k[u]$-comodules, endowed with its symmetric tensor product induced
by the cotensor product of comodules, this defines
a symmetric monoidal triangulated functor
\begin{equation*}
P: \cD(\Lambda)\too\cD(k[u]\text{-}\mathrm{Comod})\, .
\end{equation*}
The unit object of $\cD(k[u]\text{-}\mathrm{Comod})$ is $k[u]$ and for any
dg category $\cA$ we have a canonical isomorphism in the derived
category of $k$-modules
\begin{equation*}
\HP(\cA)\simeq\bbR\Hom(k[u],P(C(\cA)))\, ,
\end{equation*}
where $\HP(\cA)$ is the periodic cyclic homology complex of $\cA$.
In particular, we have a natural isomorphism of $k$-vector spaces
\begin{equation*}
\HP_n(\cA)\simeq\Hom_{\cD(k[u]\text{-}\mathrm{Comod})}(k[u],P(C(\cA)[-n]))\,.
\end{equation*}
Note that for any $k[u]$-comodule $M$ we have
$$\Hom_{\cD(k[u]\text{-}\mathrm{Comod})}(k[u],M[n])=
\begin{cases}
\Hom_{\cD(k[u]\text{-}\mathrm{Comod})}(k[u],M)&\text{if $n$ is even,}\\
\Hom_{\cD(k[u]\text{-}\mathrm{Comod})}(k[u],M[1])&\text{if $n$ is odd.}
\end{cases}$$
We thus have a canonical symmetric monoidal functor
from $\cD(k[u]\text{-}\mathrm{Comod})$ to $\mathrm{Vect}_{\bbZ/2}(k)$. By pre-composing it with the mixed complex functor $C$ we thus get
a symmetric monoidal functor
\begin{equation}\label{eq:P4}
\HP_\ast:\dgcat(k)\too \mathrm{Vect}_{\bbZ/2}(k)\,.
\end{equation}
Thanks to Kassel's work, \eqref{eq:P4} becomes
strongly symmetric monoidal when restricted to smooth and proper dg categories;
see~\cite[Theorem~7.2]{Galois}.
%
\end{example}
\section{Perioditization of classical Chow motives}\label{sec:periodizationChow}
Let $\cC$ be an additive symmetric monoidal category $\cC$ endowed with a $\otimes$-invertible
object $L$ (\ie such that the functor $X\mapsto L\otimes X$ is an equivalence of categories),
such that the twist $\tau : L\otimes L\to L\otimes L$ is the identity. Out of this data we define the category $\cC^\per$ as follows. The objects are the same as those of $\cC$ and the maps are given by the formula
\begin{equation*}
\Hom_{\cC^\per}(X,Y)=\bigoplus_{n\in\mathbf{Z}}\Hom_\cC(X,L^{\otimes n}\otimes Y)\, .
\end{equation*}
The composition of two maps $u:X\to L^{\otimes m}\otimes Y$ and $v:Y\to L^{\otimes n}\otimes Z$ of $\cC^\per$ is by definition
$$X\xrightarrow{ \ u \ }L^{\otimes m}\otimes Y\xrightarrow{L^{\otimes n}\otimes v}
L^{\otimes m}\otimes L^{\otimes n}\otimes Z
\simeq L^{\otimes m+n}\otimes Z\,.$$
Note that we have a canonical isomorphism $\unit\simeq L$ in $\cC^\per$ given by the
identity of the unit object $\unit$ in $\cC$. The category $\cC^\per$ has a symmetric
monoidal structure (this is where we use the assumption that the twist $\tau$ is the identity) which is uniquely determined by the fact that
we have a strict symmetric monoidal functor
\begin{equation}\label{eq:mapsper2}
c:\cC\too\cC^\per
\end{equation}
defined as the identity on objects and by the canonical inclusions
$$\Hom_\cC(X,Y)=\Hom_\cC(X,L^{\otimes 0}\otimes Y)\subset
\bigoplus_{n\in\mathbf{Z}}\Hom_\cC(X,L^{\otimes n}\otimes Y)=\Hom_{\cC^\per}(X,Y)\,.$$
It is easy to check that the functor \eqref{eq:mapsper2} has the following universal property:
given any symmetric monoidal category $\cD$, composing with $c$ induces
an equivalence of categories between the category of strong symmetric
monoidal functors from $\cC^\per$ to $\cD$ and the category of pairs $(F,i)$,
where $F:\cC\to\cD$ is a strong symmetric monoidal functor and $i:\unit\stackrel{\sim}{\to} F(L)$
is an isomorphism.
\begin{example}\label{example:periodisationofgradedvectorspaces}
In the case where $\cC=\mathrm{Vect}_\mathbf{Z}(k)$ is the category of $\mathbf{Z}$-graded vector spaces over some field $k$, and $L$ is a graded
vector space of rank $1$ concentrated in degree $2$, $\cC^\per$ is (canonically equivalent to) the category $\mathrm{Vect}_{\mathbf{Z}/2}(k)$
of super vector spaces (once we choose a generator of $L$).
\end{example}
Given a field $k$, let $\Chow(k)_\mathbf{Q}$ be the $\mathbf{Q}$-linear
category of \emph{Chow motives}; see \cite[Chapter~4]{andre}, \cite[Chapter~16]{Fulton} or \cite{Manin} for instance).
This category is defined as the idempotent completion of the
category whose objects are the pairs $(X,m)$, with $X$ a smooth and proper
$k$-scheme, and whose morphisms are given by
$$\Hom_{\Chow(k)_\bbQ}((X,m),(Y,n))=\prod_{i\in I}\CH^{d_i-m+n}(X_i\times Y)\,.$$
Here, $(X_i)_{i\in I}$, is the family of connected components of $X$, $d_i$ the dimension of $X_i$, and $\CH^c(W)_\bbQ$ the $\bbQ$-vector
space of algebraic cycles of codimension $c$ in $W$ modulo rational
equivalence\footnote{The category of Chow motives is often made out of
smooth and projective $k$-schemes, but this gives equivalent categories, at least
in the case where the ground field $k$ is perfect. In general, the theory of weights
tells us that the good theory of Chow motives should be constructed out of proper
and regular $k$-schemes (as opposed to proper and smooth ones),
but we won't go in this direction here.}.
The composition law is the usual composition of algebraic correspondences.
We will denote by $\cP(k)$ the category of smooth and proper $k$-schemes.
We have a natural functor
\begin{equation}\label{eq:chowmotiveofschemes}
M:\cP(k)^\op\too\Chow(k)_\bbQ \quad X\longmapsto M(X)=(X,0)
\end{equation}
which sends a map $f:X\to Y$ to the cycle of its graph.
There is a unique symmetric monoidal structure on $\Chow_\bbQ(k)$
such that the functor $M$ is strictly symmetric monoidal with respect to the
cartesian product in $\cP(k)$.
Let us write $L$ for the Lefschetz
motive (so that the motive $M(\bbP^1)$ is canonically
isomorphic to $\mathbf{Q}\oplus L$, where $\mathbf{Q}$ stands for the
motive of $\mathrm{Spec}(k)$). For any smooth and proper $k$-scheme $X$
and any integer $m$, we have a natural isomorphism $(X,m)\otimes L\simeq (X,m+1)$; this is a reformulation of the projective bundle formula.
We have also a universal perioditization functor with respect to $L$
\begin{equation}\label{eq:mapsper3}
\Chow(k)_\mathbf{Q}\too\Chow(k)^\per_\mathbf{Q}
\end{equation}
which is the identity on objects. Since this functor is symmetric
monoidal and every object of $\Chow(k)$ is rigid, all the objects of
$\Chow(k)_\mathbf{Q}^\per$ are rigid as well.

Recall that a \emph{Weil cohomology} is an additive
strong symmetric monoidal functor
\begin{equation}\label{eq:weilcoh}
H^\ast:\Chow(k)_\mathbf{Q}\too \mathrm{Vect}_\mathbf{Z}(K)
\end{equation}
such that $H^\ast(L)$ is concentrated in degree $2$
(for some field of characteristic zero $K$); see \cite[Proposition 4.2.5.1]{andre}.
Note that this implies that the vector space $H^\ast(L)$ is of dimension $1$.
In particular, by considering the perioditization of $\mathrm{Vect}_\mathbf{Z}(K)$
with respect to $H^\ast(L)$, we get a canonical commutative diagram
of strong symmetric monoidal functors
\begin{equation}\label{eq:periodizeWeilcoh}
\begin{split}
\xymatrix{
\Chow(k)_\mathbf{Q}\ar[r]^{H^\ast}\ar[d]&\mathrm{Vect}_\mathbf{Z}(K)\ar[d]\\
\Chow(k)^\per_\mathbf{Q} \ar[r]^{H^\per}&\mathrm{Vect}_{\mathbf{Z}/2}(K)
}
\end{split}
\end{equation}
in which $H^\per$ consists of the even and the odd part of $H^\ast$, \ie
$H^\per(X)=(H^\mathrm{even}(X),H^\mathrm{odd}(X))$ with
$$H^\mathrm{even}(X)=\bigoplus_{\text{$n$ even}}H^n(X)\quad \text{and} \quad
H^\mathrm{odd}(X)=\bigoplus_{\text{$n$ odd}}H^n(X) \, .$$
The known examples of Weil cohomologies include \'etale $\ell$-adic cohomology for any
prime number $\ell$ distinct from the characteristic of $k$,
crystalline cohomology (if $\mathit{char}(k)>0$), algebraic de Rham
cohomology (if $\mathit{char}(k)=0$), and Betti (or singular) cohomology
(if $k\subset\mathbf{C}$).

The category $\Chow(k)^\per_\mathbf{Q}$ is strongly related with the category
$\Hmo_0(k)$ as follows. Let $\Hmo_{0,\mathbf{Q}}(k)$ be the idempotent
completion of the universal $\mathbf{Q}$-linear additive category obtained from $\Hmo_0(k)$.
In other words, $\Hmo_{0,\mathbf{Q}}(k)$ is the idempotent completion of the
category $\Hmo_{0}(k)\otimes\mathbf{Q}$ whose objects are those of $\Hmo_0(k)$
and whose groups of morphisms are given by the formula
$$\Hom_{\Hmo_0(k)\otimes\mathbf{Q}}(\cA,\cB)=\Hom_{\Hmo_0(k)}(\cA,\cB)\otimes\mathbf{Q}\, .$$

Given a smooth and proper morphism $X\to S$, let us denote by $\todd_{X/S}$
(or simply by $\todd_X$ if the base $S$ is fixed) the
Todd class of the tangent bundle of $X/S$.
Given a commutative $\bbQ$-algebra $A$ and a power series of the form
$$\varphi=1+\sum_{n\geq 1}a_n X^n\in A[[X]]\, ,$$
its square root $\sqrt{\varphi}$ is defined as $\exp\Big(\frac{1}{2}\log(\varphi)\Big)$. In particular, this defines the square root of a Todd class.

A categorical reformulation of the Grothendieck-Riemann-Roch theorem
for proper morphisms between smooth and proper $k$-schemes
(using the explicit description of the category $\mathbf{P}(k)$
given at the end of Section \ref{sec:category-kernels}) is the following.
\begin{theorem}\label{thm:categoricalGRR}
There is a unique strong symmetric monoidal fully faithful functor
\begin{equation}\label{eq:embedChowinNChow}
\iota:\Chow(k)^\per_\mathbf{Q}\too \Hmo_{0,\mathbf{Q}}(k)
\end{equation}
with the following three properties:
\begin{itemize}
\item[(i)] it sends the motive $M(X)$ of a smooth and proper $k$-scheme $X$ to
the dg category $\perf(X)$;
\item[(ii)] for any two smooth and projective $k$-schemes $X$ and $Y$, the
monoidal constraint
$$\iota(M(X))\otimes\iota(M(Y))=\perf(X)\otimes\perf(Y)\too\perf(X\times Y)=\iota(M(X)\otimes M(Y))$$
is the map induced by the canonical isomorphism \eqref{eq:isom00};
\item[(iii)] on maps between motives of smooth and projective $k$-schemes, the above functor is given by the inverse of the isomorphism
$$\chern\cdot\sqtodd_{X\times Y}: K_0(X\times Y)\otimes\mathbf{Q}
\xrightarrow{ \sim }\bigoplus_{n\in\mathbf{Z}}\CH^n(X\times Y)_\mathbf{Q}\,.$$
\end{itemize}
\end{theorem}
\begin{proof}
Given a smooth and proper $k$-scheme $X$, let us write
$$\CH^\ast(X)=\bigoplus_{n\in\bbZ}\CH^n(X)_\bbQ=\prod_{n\in\bbZ}\CH_\bbQ^n(X)\, .$$
Consider now a smooth and proper $k$-scheme $S$. Given two smooth and proper $S$-schemes $X$ and $Y$, let us denote by
$p:X\times_S Y\to X$ and $q:X\times_S Y\to Y$ the canonical projections.
For any $\cE\in K_0(X\times_S Y)$
we then have the following commutative diagram of abelian groups
(in which $\todd_q=p^\ast(\todd_X)$
denotes the Todd class of the tangent bundle of $q$) 
\begin{equation}\label{eq:proofcatGRR1}\begin{split}\xymatrix{
K_0(X)\ar[r]^(.45){p^\ast}\ar[d]^{\chern\cdot\sqtodd_X}
&K_0(X\times_S Y)\ar[r]^{\cdot\cE}\ar[d]^{\chern\cdot p^\ast(\sqtodd_{X})}
&K_0(X\times_S Y)\ar[r]^(.55){q_\ast}\ar[d]^{\chern\cdot\todd_q\cdot q^\ast(\sqtodd_{Y})}
&K_0(Y)\ar[d]^{\chern\cdot\sqtodd_Y}\\
\CH^\ast(X)\ar[r]_(.45){p^\ast}
&\CH^\ast(X\times_S Y)\ar[r]_{\cdot\chern(\cE)\cdot \sqtodd_{X\times_S Y}}
&\CH^\ast(X\times_S Y)\ar[r]_(.55){q_\ast}
&\CH^\ast(Y)
}\end{split}\end{equation}
The commutavity of the left-hand and middle squares is obvious, while
the commutativity of the right-hand-side square stems from the Grothendieck-Riemann-Roch
theorem (see \cite[Theorem 15.2]{Fulton}, for instance),
and from the projection formula $q_\ast(a)\cdot b=q_\ast(a\cdot q^\ast(b))$.

Now, let $\cC$ denote the category whose objects are the smooth and proper $k$-schemes, whose morphisms are defined by the formula
$$\Hom_\cC(X,Y)=\CH^\ast(X\times Y)\,,$$
and whose composition law is defined through the usual formulas
for correspondences. The above arguments imply that there is a unique functor
\begin{equation}\label{eq:proofcatGRR2}
\iota:\cC\too\Hmo_{0,\bbQ}(k) \quad X\longmapsto\perf(X)
\end{equation}
which is defined on maps by the inverse of the isomorphisms $\chern\cdot\sqtodd_{X\times Y}$. 
Indeed, the compatibility with the composition of maps has already been checked
above in disguise: let $X$ and $Y$ be smooth and proper $k$-schemes and $\cE\in K_0(X\times Y)$. If $\pi:X\times Y\times Z\to X\times Y$ denotes the canonical projection,
one just has to put $S=Z$ and to replace $X$ by $X\times Z$, $Y$ by $Y\times Z$ and $\cE$ by $\pi^\ast(\cE)$ in the commutative diagram \eqref{eq:proofcatGRR1}. The proof of this claim follows from the description of the composition law
of maps between objects of the form $\perf(X)$ in $\Hmo_{0,\bbQ}(k)$, from the
explicit description of the category of integral kernels (see Section~\ref{sec:category-kernels}),
and from the Yoneda lemma. It is clear that the functor \eqref{eq:proofcatGRR2}
is strongly symmetric monoidal, fully faithful, and uniquely determined by
the isomorphisms $\chern\cdot\sqtodd_{X\times Y}$.
Since one can describe $\Chow^\per_\bbQ(k)$ as the idempotent completion of $\cC$
and since $\Hmo_{0,\bbQ}(k)$ is idempotent complete, the proof of the theorem now follows immediately from pure ``abstract nonsense''.
\end{proof}

Let $\bbP_{0,\bbQ}(k)$ be the full subcategory of $\Hmo_{0,\bbQ}(k)$
whose objects are those isomorphic to a direct factor of a dg category
of the form $\perf(X)$, with $X$ a proper and smooth $k$-scheme.
Theorem~\ref{thm:categoricalGRR} might be equivalently formulated as an explicit
equivalence of symmetric monoidal categories from $\Chow^\per_\bbQ(k)$
to $\bbP_{0,\bbQ}(k)$. In particular, for any Weil cohomology $H^\ast$, its
perioditization $H^\per$ may as well be seen as a symmetric monoidal additive
functor
\begin{equation}\label{eq:defFMChowper0}
H^\per:\bbP_{0,\bbQ}(k)\to\mathrm{Vect}_{\bbZ/2}(K)\, .
\end{equation}
Moreover, the explicit description of the category $\bbP(k)$
of integral kernels (Section \ref{sec:category-kernels})
implies that the category $\bbP_{0,\bbQ}(k)$
is canonically equivalent to the idempotent completion of the
category whose objects are smooth and proper $k$-schemes, and whose maps are given by the groups $K_0(X\times Y)\otimes\bbQ$
(with the usual composition law for correspondences, using pullback
and pushforward maps in $K$-theory).

The explicit description of the functor \eqref{eq:defFMChowper0} is the following. For any two proper and smooth $k$-schemes $X$ and $Y$,
given an element $\cE\in K_0(X\times Y)$, we will write
\begin{equation}\label{eq:defFMChowper1}
\Phi_\cE=\chern(\cE)\cdot\sqtodd_{X\times Y}\in\bigoplus_{n\in\bbZ}\CH^n(X\times Y)_\bbQ\, .
\end{equation}
Given a Weil cohomology $H^\ast$ as above, as its perioditization $H^\per$
defines a functor from $\Chow^\per_\bbQ(k)$ to the category of super vector spaces,
we obtain a super map
\begin{equation}\label{eq:defFMChowper2}
H^\per(\Phi_\cE):H^\per(X)\too H^\per(Y)\,.
\end{equation}
Its trace is computed by the formula
\begin{equation}\label{eq:defFMChowper3}
\mathrm{Tr}\, H^\per(\Phi_\cE)=
\mathrm{Tr}\,  H^{\mathrm{even}}(\Phi_\cE) - \mathrm{Tr} \, H^{\mathrm{odd}}(\Phi_\cE)
\in K=\mathrm{End}_{\mathrm{Vect}_{\bbZ/2}}(K)\, .
\end{equation}
\section{Noncommutative Lefschetz Theorem}\label{section:nclefschetz}
\begin{lemma}\label{lemma:abstracttraces}
Let $\cC$ be symmetric monoidal additive category with unit object $\unit$
such that $\mathrm{End}_\cC(\unit)\subset\mathbf{Q}$.
Consider a strongly symmetric monoidal additive functor $F:\cC\to\mathsf{D}$.
Then, for any rigid object $X$ of $\cC$, the object $L(X)$ is also rigid, and
for any map $f:X\to X$ of $\cC$, the trace $\mathrm{Tr}(L(f))$ is equal to the image
of the trace $\mathrm{Tr}(f)$ under the morphism of rings $\mathrm{End}_\cC(\unit)\to \mathrm{End}_\mathsf{D}(\unit)$ induced by $L$. In particular, $\mathrm{Tr}(L(f))$ is the image of a rational number
(more precisely, of an element of $\mathrm{End}_\cC(\unit)$) which does
not depend on the functor $L$.
\end{lemma}
The proof of the preceding lemma is straightforward.

\begin{proposition}\label{prop:absolutetracencmot}
Let $\cA$ be a smooth and proper dg category. Consider
a perfect dg $\cA$-bimodule $M$ (\ie $M \in \cD_c(\cA^\op\otimes^\bbL \cA)$), and
let us denote by $\Phi_{[M]}: \cA \to \cA$ the associated endomorphism
of $\cA$ in $\NChow(k)$. Then we have the formula
$$\mathrm{Tr}(\Phi_{[M]}) = [\HH(\cA;M)] \in K_0(\cD_c(k))\,.$$
In particular if $k$ is a local ring we have the identification $K_0(k)=K_0(\cD_c(k))\simeq\mathbf{Z}$ and consequently the formula
$$ \mathrm{Tr}(\Phi_{[M]}) = \sum_i (-1)^i \mathrm{rk}\, \HH_i(\cA;M)\,.$$
\end{proposition}

\begin{proof}
See \cite[Prop.~4.3]{Semi}.
\end{proof}
\begin{corollary}\label{cor:absolutetracencmotcomm}
Given a smooth and proper $k$-scheme $X$ and a bounded complex $\cE$ of
coherent $\cO_{X\times X}$-modules, let us denote by
$\Phi_\cE:\perf(X)\to\perf(X)$ the Fourier-Mukai morphism in $\Hmo_0(k)$
induced by $\cE$. Then, the following equality holds
$$\mathrm{Tr}(\Phi_{\cE}) = [\HH(\cE)] \in K_0(\cD_c(k))\,.$$
\end{corollary}
\begin{proof}
This follows automatically from the combination of Proposition~\ref{prop:absolutetracencmot} with
the identification $\HH(\perf(X);M)\simeq\HH(\cE)$
where $M$ denotes the $\perf(X)$-bimodule defined by $\cE$
via the isomorphism \eqref{eq:isom00}; recall from Example \ref{example:HH} that thanks to the work of Keller all versions of Hochschild homology of schemes agree (including
with coefficients).
\end{proof}
%

\subsection*{Proof of Theorem~\ref{thm:main}}\label{proof:thm:main}
Let $L:\dgcat\to\mathsf{D}$ be a symmetric monoidal invariant and $\underline L:\Hmo_0\to\mathsf{D}$
the corresponding symmetric monoidal additive functor; see Proposition \ref{prop:universalmonoidal}.
Theorem~\ref{thm:main} follows immediate from Lemma \ref{lemma:abstracttraces} applied to $\underline L$
and from Proposition \ref{prop:absolutetracencmot}.\qed

\subsection*{Proof of Theorem~\ref{thm:main2}}\label{proof:thm:main2}
Assume that $k$ is a field and consider a Weil cohomology
$$H^\ast:\Chow_\mathbf{Q}(k)\too\mathrm{Vect}_\mathbf{Z}(K)\, .$$
Recall from \eqref{eq:periodizeWeilcoh} the construction of the perioditization functor
\begin{equation}\label{eq:H-per}
H^\per:\Chow_\bbQ^\per(k)\too\mathrm{Vect}_{\bbZ/2}(K)\, .
\end{equation}
Let $\cE$ be a bounded complex of coherent $\cO_{X\times X}$-modules and $\Phi_\cE$ be the corresponding Fourier-Mukai
endomorphism in the category $\Chow^\per_\bbQ(k)$; see \eqref{eq:defFMChowper1}.
Thanks to Theorem \ref{thm:categoricalGRR} and
Corollary \ref{cor:absolutetracencmotcomm} the
trace of $\Phi_\cE$ in $\Chow^\per_\bbQ(k)$ is equal to the Euler
character of the Hochschild homology complex $\HH(\cE)$. The proof follows then from the application of Lemma \ref{lemma:abstracttraces} to \eqref{eq:H-per} and from the Formula \eqref{eq:defFMChowper3}.\qed

\subsection*{Proof of Lunts' Theorem \ref{thm:Lunts1}}
If $k$ is a field, by applying Theorem~\ref{thm:main} for $L=\HH$, the Hochschild
homology functor (Example~\ref{example:HH}), we see that
equality \eqref{eq:equality3} is a particular case of
Theorem \ref{thm:main}. Similarly, by applying Theorem \ref{thm:main2} in the
case of Betti cohomology with rational (or complex) coefficients,
we get Equality \eqref{eq:equality2}. 
Equality \eqref{eq:equality1}
follows immediately from Equality \eqref{eq:equality3}
(and from Corollary \ref{cor:absolutetracencmotcomm}).\qed

\begin{remark}
Lunts derives Formula \eqref{eq:equality2}
from Formula \eqref{eq:equality1} using the work of Macri-Stellari \cite[Theorem~1.2]{MS} on the  Hochschild-Kostant-Rosenberg
isomorphism; see \cite[Theorem~1.4]{Lunts}. In other words, Lunts' proof of \eqref{eq:equality2} first considers Betti cohomology as a strong symmetric monoidal
functor from the category of Hochschild correspondences between smooth
and projective $k$-schemes (instead of the category of rational
$K$-theory correspondences) and then applies some form of Lemma \ref{lemma:abstracttraces}.
This argument requires to check that
the HKR isomorphism is functorial with respect to Fourier-Mukai functors, which is precisely the content of \cite[Theorem~1.2]{MS}. Nervertheless, note that Lunts' argument uses some
weak form of Theorem \ref{thm:categoricalGRR} though, at least to define
the functoriality with respect to Fourier-Mukai functors for (periodic) Betti
cohomology.
\end{remark}
\section{Noncommutative Hirzebruch-Riemman-Roch}\label{sec:HRR}

\subsection*{Proof of Theorem~\ref{thm:HRR1}}\label{proof:thm:HRR1}
Given a proper dg category $\cA$, we have a well-defined pairing in $\Hmo(k)$:
\begin{equation}\label{eq:absolutepairing0}
\cA\otimes^\bbL \cA^\op \to \perf(k)\simeq k \ , \quad
(x,y) \mapsto \cA(y,x)\,.
\end{equation}
By applying the symmetric monoidal
additive functor $\Hom_{\Hmo(k)}(k,-)$, and using the identifications $K_0(\cA)=\Hom_{\Hmo(k)}(k,\cA)$, we then get a pairing of abelian groups:
\begin{equation}\label{eq:absolutepairing1}
\langle - , - \rangle:K_0(\cA)\otimes K_0(\cA)\to K_0(k)\, .
\end{equation}
Now, let $L:\dgcat(k) \to \mathsf{D}$ be
a strong symmetric monoidal additive invariant.
Thanks to Proposition \ref{prop:universalmonoidal} the functor $L$
canonically defines a strong symmetric monoidal functor $\underline{L}:\Hmo_0(k)\to\mathsf{D}$. 
For a dg category $\cC$, we will write
\begin{equation*}
H(\cC)=\Hom_\mathsf{D}(\unit,L(\cC))=\Hom_\mathsf{D}(\unit,\underline{L}(\cC))\, .
\end{equation*}
Since the functor $\underline{L}$ is strongly unital it induces
functorial maps
\begin{equation}\label{eq:abtractChern}
\chern_L:K_0(\cC)\too H(\cC)
\end{equation}
which we call the \emph{Chern characters}.

By applying the strong tensor functor $\underline{L}$ to \eqref{eq:absolutepairing0},
we obtain a pairing
\begin{equation}\label{eq:pair}
\underline{L}(\underline{\cA}): L(\cA) \otimes L(\cA^\op) \too {\bf 1}
\end{equation}
and consequently, by further applying the symmetric monoidal additive functor $\Hom_\mathsf{D}(\unit,-)$, a canonical pairing of abelian groups
\begin{equation}\label{eq:pair2}
\langle - , - \rangle: H(\cA) \otimes H(\cA^\op) \too H(k)=\mathrm{End}_\mathsf{D}(\unit)\,.
\end{equation}
The fact that the functor $\underline{L}$ is strongly symmetric monoidal implies that the pairings \eqref{eq:absolutepairing1}
and \eqref{eq:pair2} are compatible with the Chern character \eqref{eq:abtractChern}.
In other words, we have the following commutative diagram:
\begin{equation}\label{eq:formalncHRR}
\begin{split}
\xymatrix{
K_0(\cA)\otimes K_0(\cA^\op)\ar[r]^(.65){\langle - , - \rangle}
\ar[d]_{\chern_L\otimes\chern_L}& K_0(k)\ar[d]^{\chern_L}\\
H(\cA) \otimes H(\cA^\op)\ar[r]^(.65){\langle - , - \rangle}& H(k)
}\end{split}
\end{equation}
Given a perfect $\cA$-module $M$, which we can see as a map $M:k\simeq\perf(k)\to\cA$ in $\Hmo(k)$,
we will still denote by $M$
its class in $K_0(\cA)$. This can (and should) be interpreted
as a map from $k\simeq\perf(k)$ to $\cA$ in $\Hmo_0(k)$.
By composing the duality functor $\perf(k)^\op\to\perf(k)$
with $M:\perf(k)\to\cA$ we obtain 
a map $\perf(k)^\op\to\cA$ and consequently a map
$$DM:k\simeq\perf(k)\too \cA^\op$$
(in other words, $DM$ is a perfect $\cA^\op$-module).
Note that, as we work up to Morita equivalence, we may always
suppose that $M$ is quasi-isomorphic to an $\cA$-module of the
form $\cA(-,x)$, with $x$ an object of $\cA$, in which case $DM$
simply corresponds to the $\cA^\op$-module $\cA(x,-)=\cA^\op(-,x)$. In particular,
the pairing \eqref{eq:absolutepairing0} applied to the pair $(N,DM)$, with $M$ and $N$ any two perfect $\cA$-modules, gives the perfect $k$-module
$\bbR\mathit{Hom}_\cA(M,N)$.
We then have two maps $N:k\to \cA$ and $DM:k\to\cA^\op$ in $\Hmo_0(k)$,
and the commutativity of the square \eqref{eq:formalncHRR} gives the formula
$$\chern_L(\langle N , DM \rangle)=\langle \chern_L(N) , \chern_L(DM) \rangle\, .$$
But, by definition, $\langle N , DM \rangle$ is the class of the complex
obtained by applying the pairing \eqref{eq:absolutepairing0} to $N$ and $DM$.
We then obtain Equality \eqref{eq:HRR1}.\qed
\subsection*{Proof of Theorem~\ref{thm:HRR}}\label{proof:thm:HRR}
In the case where $L=\HH$ is Hochschild homology, the
Chern characters \eqref{eq:abtractChern} agree with Dennis trace maps; see \cite[Theorem~2.8]{Prods}.
Therefore our constructions are compatible with those of Shklyarov \cite{Shklyarov}: Theorem \ref{thm:HRR}
readily follows from Theorem \ref{thm:HRR1}.\qed

\begin{example}
When $L$ is the mixed complex functor $C$ of Example~\ref{example:mix}
and $A$ is a proper dg algebra $A$, equality \eqref{eq:HRR1} reduces to 
$$ \chern^-(\bbR\mathit{Hom}_\cA(M,N)) = \langle \chern^-(N), \chern^-(DM)\rangle\,,$$
where $\chern^-$ is the classical Chern character with values in negative cyclic homology; see \cite[Theorem~2.8]{Prods}.
\end{example}
\begin{remark}
Given a proper dg category $\cA$ and two perfect dg $\cA$-modules $M$ and $N$, the following equality holds (essentially by definition)
\begin{equation}\label{eq:last1}
\HH(\cA;DM\otimes^\bbL_k N)=\bbR\mathit{Hom}_\cA(M,N)\,.
\end{equation}
Hence, whenever $\cA$ is smooth
and proper, Shklyarov's Theorem~\ref{thm:HRR} is a direct consequence of
Lunts' Formula \eqref{eq:equality3}.
%
\end{remark}

\end{document}